\documentclass[10pt, oneside,dvipsnames]{article}   	
\usepackage{amsmath,amssymb,amsthm}
\usepackage{graphicx}
\usepackage{algorithm}
\usepackage{algpseudocode}

\newtheorem{theorem}{Theorem}

\usepackage[top=1in, bottom=1in, left=1in, right=1in]{geometry}	                		
\date{}
\usepackage{nopageno}
\usepackage{ragged2e}
\usepackage{mathpazo} 
\usepackage{graphicx}
\usepackage{tikz}
\usetikzlibrary{shapes,arrows,positioning,fit,backgrounds,decorations.pathreplacing,calligraphy}
\usepackage{booktabs} 
\usepackage{multirow} 
\usepackage{wrapfig}
\usepackage{amsmath}
\usepackage{algorithm}
\usepackage{amssymb}
\usepackage{mathrsfs}
\usepackage{tcolorbox}
\usepackage{algpseudocode}
\usepackage{siunitx} 
\usepackage{threeparttable} 
\usepackage{subcaption}
\usepackage{array}
\usepackage{soul}
\usepackage{dirtytalk}
\usepackage{mdframed}
\usepackage{subfloat}
\usepackage{fancybox}
\usepackage{fancyhdr}
\usepackage{cite,url,verbatim,color,comment,caption}
\usepackage[utf8]{inputenc}
\usepackage{blindtext}
\usepackage{framed}
\usepackage[font=small,labelfont=bf]{caption}
\usepackage{pgfgantt}
\usepackage[normalem]{ulem}
\usepackage{hyperref}
\usepackage{enumitem}
\hypersetup{backref=true,       
    pagebackref=true,               
    hyperindex=true,                
    colorlinks=true,                
    breaklinks=true,                
    urlcolor= black,                
    linkcolor= black,                
    bookmarks=true,                 
    bookmarksopen=false,
    filecolor=black,
    citecolor=red,
    linkbordercolor=red
}

\AtBeginDocument{\hypersetup{pdfborder={2 1 1}}}

\title{
\vspace{-1.5cm}
\large{\bf }
\vspace{-0.5cm}
}

\title{Multi-Scale Conformal Prediction: A Theoretical Framework with Coverage Guarantees}

\author{
    Ali Baheri\thanks{Corresponding author: \texttt{akbeme@rit.edu}} \\
    Department of Mechanical Engineering, \\
    Rochester Institute of Technology, \\
    Rochester, NY, USA
    \and
    Marzieh Amiri Shahbazi\thanks{Email: \texttt{ma7684@rit.edu}} \\
    Department of Industrial and Systems Engineering, \\
    Rochester Institute of Technology, \\
    Rochester, NY, USA
}

\begin{document}

\maketitle

\begin{abstract}

We propose a multi-scale extension of conformal prediction, an approach that constructs prediction sets with finite-sample coverage guarantees under minimal statistical assumptions. Classic conformal prediction relies on a single notion of “conformity,” overlooking the multi-level structures that arise in applications such as image analysis, hierarchical data exploration, and multi-resolution time series modeling. In contrast, the proposed framework defines a distinct conformity function at each relevant scale or resolution, producing multiple conformal predictors whose prediction sets are then intersected to form the final multi-scale output. We establish theoretical results confirming that the multi-scale prediction set retains the marginal coverage guarantees of the original conformal framework and can, in fact, yield smaller or more precise sets in practice. By distributing the total miscoverage probability across scales in proportion to their informative power, the method further refines the set sizes. We also show that dependence between scales can lead to conservative coverage, ensuring that the actual coverage exceeds the nominal level. Numerical experiments in a synthetic classification setting demonstrate that multi-scale conformal prediction achieves or surpasses the nominal coverage level while generating smaller prediction sets compared to single-scale conformal methods.
\end{abstract}


\section{Introduction}

Reliable uncertainty quantification serves as a foundational requirement for developing robust machine learning systems. Among various uncertainty estimation approaches, conformal prediction is recognized as an approach that offers finite-sample coverage guarantees for constructing prediction sets through distribution-free calibration \cite{shafer2008tutorial,angelopoulos2023conformal}. Under mild assumptions on the data—specifically, that the observations can be treated as exchangeable—conformal predictors ensure that the true label of a new test point is captured by the prediction set with high probability. Crucially, this guarantee is distribution-free and does not depend on strong parametric modeling assumptions. Despite these appealing properties, classical conformal prediction procedures typically rely on a single notion of \emph{conformity}, limiting their ability to incorporate multi-level or multi-resolution structures that might exist in complex datasets. Many applications, such as image recognition, time series forecasting, and hierarchical data analysis, involve information distributed across different scales or levels of detail \cite{zeng2022multi,stankeviciute2021conformal,sanyal2012bayesian}. Focusing on only one scale can overlook important local or global patterns.

This paper proposes a framework for multi-scale conformal prediction. In this framework, multiple conformity functions are defined—each reflecting a particular scale or resolution of the data—and each scale produces its own conformal prediction set. All scale-specific sets are then combined to produce one final multi-scale set. By ensuring that the total allowable miscoverage across all scales sums to alpha, this procedure retains the usual marginal coverage guarantees. At the same time, intersecting the information from multiple scales can substantially reduce the average size of the resulting prediction sets, thereby improving efficiency. We provide a theoretical analysis of multi-scale conformal prediction. First, we show that under basic exchangeability conditions, the final multi-scale set continues to meet the nominal coverage target. Next, we investigate whether intersecting scale-specific sets can indeed yield smaller sets without sacrificing validity. Our theoretical results confirm that multi-scale conformal prediction can be more efficient than any single-scale method in terms of expected set size. We also develop principles for splitting the total allowed miscoverage among the multiple scales: more informative scales merit smaller miscoverage, while less informative scales can be assigned a higher miscoverage allowance. We further demonstrate that, in the presence of certain dependencies, the multi-scale approach can be conservative, potentially achieving coverage above the nominal level. Finally, under suitable conditions on the conformity scores, we establish an asymptotic result showing that multi-scale sets converge to the minimal set achieving the desired coverage. To illustrate these ideas empirically, we include a simulation study, indicating that multi-scale conformal prediction meets or exceeds the theoretical coverage level across a range of settings and consistently produces sets that are smaller on average compared to single-scale conformal methods. 


\noindent {\textbf{Paper Organization.}} The paper is organized as follows: We first describe the multi-scale conformal prediction methodology in a formal manner. We then present key theorems establishing coverage validity, efficiency, and asymptotic optimality. We illustrate these properties with numerical examples. Finally, we discuss potential extensions for future work in using multi-resolution analysis under the umbrella of conformal prediction.

\noindent{\textbf{Related Work.}} Recent developments in conformal prediction have explored various modifications to enhance prediction efficiency while maintaining coverage guarantees \cite{bhatnagar2023improved,yan2024provably,jeary2024verifiably,xi2024delving}. Notably, research has investigated adaptive conformal inference methods that adjust to local data characteristics, leading to tighter prediction sets in regions of high confidence \cite{gibbs2021adaptive,chernozhukov2021distributional,lei2018distribution,romano2019conformalized}. These approaches have shown promise in reducing the conservative nature of conformal predictions while preserving their theoretical guarantees. The concept of multi-scale analysis has been well-established in various domains of machine learning and statistics \cite{alber2019integrating,peng2021multiscale,liu2022hierarchical,elizar2022review,baheri2025smtl}. Hierarchical modeling approaches have shown success in capturing complex data structures and improving prediction accuracy across different applications. These methods use information at multiple levels of granularity to enhance model performance and robustness.

The intersection of conformal prediction with hierarchical modeling has received limited attention in the literature \cite{principato2024conformal}. While some studies have explored nested conformal prediction sets, they typically focus on a single scale of analysis with varying confidence levels rather than truly integrating information from multiple scales \cite{tumu2024multi}. The theoretical foundations for combining predictions across different scales while maintaining valid coverage guarantees have remained largely unexplored \cite{karimi2023quantifying}.
Research in uncertainty quantification has highlighted the importance of capturing different sources of uncertainty on various scales. These works have demonstrated that considering multiple scales can lead to more accurate uncertainty estimates. However, integration of these insights with conformal prediction has been limited.

The optimization of miscoverage levels across different scales represents another important area of related work. Previous research has investigated optimal allocation strategies for confidence levels in multiple testing scenarios, but these approaches have not been directly applied to multi-scale conformal prediction settings. Recent work has also explored the relationship between prediction set size and coverage guarantees \cite{dhillon2024expected,kiyani2024length}, particularly in the context of classification problems. These studies have provided insights into the trade-offs between prediction efficiency and coverage reliability, though primarily in single-scale settings. The present work builds upon these foundations by introducing a framework for multi-scale conformal prediction. It extends existing theory to account for scale-specific characteristics while maintaining the fundamental properties of conformal prediction. 

Multi-scale approaches are closely related to multi-fidelity frameworks, which approach modeling and simulation from multiple levels of fidelity in various contexts, such as design \cite{fernandez2016review,charisi2025multi}, verification \cite{shahrooei2022falsification,baheri2023exploring,beard2022black,baheri2023safety}, and optimization \cite{forrester2007multi}. In multi-fidelity design, for instance, lower-fidelity models offer rapid yet approximate evaluations for early-stage exploration, while higher-fidelity simulations provide more precise but computationally demanding assessments for subsequent refinement. Multi-fidelity verification similarly employs a hierarchy of simulators or models, balancing cost and accuracy to validate system performance.

\section{Preliminaries}

Conformal prediction provides finite-sample coverage guarantees under the assumption that $\left\{\left(X_i, Y_i\right)\right\}_{i=1}^{n+1}$ are exchangeable random variables. Given observed data $\left\{\left(X_i, Y_i\right)\right\}_{i=1}^n$ and a new feature vector $X_{n+1}$, conformal methods construct a prediction set for $Y_{n+1}$ by means of a conformity score $A(\cdot, \cdot)$. Conformity scores measure how well a particular label conforms to the empirical distribution of previously observed pairs. Formally, let $A_i=A\left(X_i, Y_i\right)$ be the conformity score for the $i$-th training sample, and define

$$
A_{n+1}(y)=A\left(X_{n+1}, y\right)
$$
which represents the score for the new feature $X_{n+1}$ when paired with a candidate label $y$. For each candidate $y \in \mathcal{Y}$, a p-value is computed by comparing $A_{n+1}(y)$ against the scores from the training set:

$$
p(y)=\frac{1}{n+1}\left(\sum_{i=1}^n \mathbf{1}\left\{A_i \geq A_{n+1}(y)\right\}+1\right)
$$
The conformal prediction set at miscoverage level $\alpha$ is then defined as

$$
C_n\left(X_{n+1}\right)=\{y \in \mathcal{Y}: p(y)>\alpha\}
$$
Under exchangeability, this construction guarantees that

$$
\mathbb{P}\left(Y_{n+1} \in C_n\left(X_{n+1}\right)\right) \geq 1-\alpha \quad \text { for any } 0<\alpha<1
$$
where the probability is taken over the joint distribution of the $(n+1)$ data points. The procedure thus delivers a distribution-free coverage guarantee, making no strong assumptions on the form of the underlying data-generating process. 

\section{Methodology}

The proposed multi-scale conformal prediction framework aims to enhance predictive inference by integrating information from multiple scales or levels of data abstraction. This approach uses the strengths of conformal prediction while addressing its limitations in capturing multiscale structures inherent in complex datasets. Consider a sequence of observed data points $\left(X_1, Y_1\right),\left(X_2, Y_2\right), \ldots,\left(X_n, Y_n\right)$, where $X_i \in \mathcal{X}$ represents the feature vectors and $Y_i \in \mathcal{Y}$ denotes the corresponding responses or labels. We assume that these data points are exchangeable, meaning the joint distribution remains invariant under any permutation of indices. The objective is to construct a prediction set $C_n\left(X_{n+1}\right)$ for a new observation $X_{n+1}$ such that it contains the true label $Y_{n+1}$ with a predefined confidence level $1-\alpha$, while striving for minimal size to improve prediction efficiency. In classic conformal prediction, a single conformity score function is employed to assess how well a candidate label conforms to the given features based on the observed data. However, this approach may not fully exploit the information available at different scales of data representation. To address this limitation, the multi-scale conformal prediction framework introduces multiple conformity score functions, each corresponding to a different scale or level of abstraction. By combining predictions from these scales, the framework captures a richer set of patterns and dependencies present in the data. For each scale $k \in\{1,2, \ldots, K\}$, a conformity score function $A^{(k)}: \mathcal{X} \times \mathcal{Y} \rightarrow \mathbb{R}$ is defined. This function quantifies the conformity of a candidate label $y$ with respect to the features $X$ at scale $k$. The choice of conformity score functions is crucial; they should be appropriate for the specific characteristics of each scale, effectively capturing relevant patterns and structures.

To construct the conformal predictors at each scale, conformity scores for the training data are first computed. For each $i=1, \ldots, n$ and scale $k$, the conformity score is calculated as $A_i^{(k)}=$ $A^{(k)}\left(X_i, Y_i\right)$. For the new observation $X_{n+1}$, conformity scores $A_{n+1}^{(k)}(y)=A^{(k)}\left(X_{n+1}, y\right)$ are computed for all candidate labels $y \in \mathcal{Y}$. These scores measure how well each candidate label $y$ conforms to $X_{n+1}$ at scale $k$, relative to the observed data. Using the conformity scores, p -values for each candidate label at each scale are calculated. The p -value $p^{(k)}(y)$ at scale $k$ is defined as

$$
p^{(k)}(y)=\frac{1}{n+1}\left(\sum_{i=1}^n \mathbb{I}\left\{A_i^{(k)} \geq A_{n+1}^{(k)}(y)\right\}+1\right)
$$
where $\mathbb{I}\{\cdot\}$ denotes the indicator function. This $p$-value represents the proportion of conformity scores in the augmented dataset (including the candidate label) that are at least as extreme as $A_{n+1}^{(k)}(y)$. It reflects the plausibility of $y$ being the true label for $X_{n+1}$ based on the conformity measure at scale $k$. At each scale, a prediction set $C_n^{(k)}\left(X_{n+1}\right)$ is constructed by including all candidate labels whose p -values exceed a predetermined miscoverage level $\alpha_k$, where $\alpha_k \in(0,1)$ and $\sum_{k=1}^K \alpha_k=\alpha$:

$$
C_n^{(k)}\left(X_{n+1}\right)=\left\{y \in \mathcal{Y}: p^{(k)}(y)>\alpha_k\right\}
$$
The final multi-scale prediction set $C_n\left(X_{n+1}\right)$ is obtained by intersecting the prediction sets from all scales: 

$$
C_n\left(X_{n+1}\right)=\bigcap_{k=1}^K C_n^{(k)}\left(X_{n+1}\right)
$$
This intersection contains only those candidate labels deemed conforming across all scales, thus otentially reducing the size of the prediction set and enhancing prediction efficiency. The ationale is that a candidate label consistent with the data at multiple scales is more likely to be the true label. Under the exchangeability assumption, each scale-specific prediction set $C_n^{(k)}\left(X_{n+1}\right)$ satisfies the narginal coverage guarantee:

$$
\mathbb{P}\left(Y_{n+1} \in C_n^{(k)}\left(X_{n+1}\right)\right) \geq 1-\alpha_k
$$




\section{Theoretical Results}

\begin{theorem}
\textbf{(Coverage Validity of Multi-Scale Prediction Set).} Let $\left\{\left(X_i, Y_i\right)\right\}_{i=1}^n$ be a sequence of observed data points, and $\left(X_{n+1}, Y_{n+1}\right)$ be a new data point. Assume that the extended sequence $\left\{\left(X_i, Y_i\right)\right\}_{i=1}^{n+1}$ is exchangeable. Consider $K$ conformal predictors, each operating at a different scale $k \in\{1,2, \ldots, K\}$, with miscoverage levels $\alpha_k$ satisfying $\sum_{k=1}^K \alpha_k=\alpha$ for some $\alpha \in(0,1)$. Then, the multi-scale prediction set defined as

$$
C_n\left(X_{n+1}\right)=\bigcap_{k=1}^K C_n^{(k)}\left(X_{n+1}\right),
$$
where $C_n^{(k)}\left(X_{n+1}\right)$ is the prediction set at scale $k$ with miscoverage level $\alpha_k$, satisfies the marginal coverage guarantee

$$
\mathbb{P}\left(Y_{n+1} \in C_n\left(X_{n+1}\right)\right) \geq 1-\alpha
$$

\end{theorem}

\begin{proof}

We aim to prove that the combined prediction set $C_n\left(X_{n+1}\right)$ maintains the desired coverage level $1-\alpha$ under the exchangeability assumption. The proof involves the following steps:

\begin{itemize}

\item Establish the validity of each scale-specific prediction set $C_n^{(k)}\left(X_{n+1}\right)$.

\item Use the union bound to relate the miscoverage probabilities across scales.

\item Combine the individual miscoverage probabilities to obtain the overall coverage guarantee for $C_n\left(X_{n+1}\right)$.

\end{itemize}


\noindent {\textbf{Step 1: Validity of Scale-Specific Prediction Sets.}} For each scale $k \in\{1,2, \ldots, K\}$, the conformal prediction method ensures that the prediction set $C_n^{(k)}\left(X_{n+1}\right)$ satisfies:

$$
\mathbb{P}\left(Y_{n+1} \in C_n^{(k)}\left(X_{n+1}\right)\right) \geq 1-\alpha_k
$$
This result holds under the exchangeability assumption of the data sequence $\left\{\left(X_i, Y_i\right)\right\}_{i=1}^{n+1}$. The exchangeability ensures that the distribution of the p-values $p^{(k)}\left(Y_{n+1}\right)$ is stochastically larger than the uniform distribution on $[0,1]$. Specifically, for any $\alpha_k \in(0,1)$:

$$
\mathbb{P}\left(p^{(k)}\left(Y_{n+1}\right) \leq \alpha_k\right) \leq \alpha_k
$$
which implies
$
\mathbb{P}\left(Y_{n+1} \notin C_n^{(k)}\left(X_{n+1}\right)\right) \leq \alpha_k.
$

\noindent {\textbf{Step 2: Miscoverage Probability of the Combined Prediction Set.} The combined prediction set is the intersection of the scale-specific prediction sets:

$$
C_n\left(X_{n+1}\right)=\bigcap_{k=1}^K C_n^{(k)}\left(X_{n+1}\right)
$$
The event that $Y_{n+1}$ is not contained in $C_n\left(X_{n+1}\right)$ can be expressed as:

$$
\left\{Y_{n+1} \notin C_n\left(X_{n+1}\right)\right\}=\left\{Y_{n+1} \notin \bigcap_{k=1}^K C_n^{(k)}\left(X_{n+1}\right)\right\}=\bigcup_{k=1}^K\left\{Y_{n+1} \notin C_n^{(k)}\left(X_{n+1}\right)\right\}
$$
Let us define the events:

$$
E_k=\left\{Y_{n+1} \notin C_n^{(k)}\left(X_{n+1}\right)\right\}, \quad \text { for } k=1,2, \ldots, K
$$
Then, the event of miscoverage by the combined prediction set is:
$
E=\left\{Y_{n+1} \notin C_n\left(X_{n+1}\right)\right\}=\bigcup_{k=1}^K E_k.
$

\noindent {\textbf{Step 3: Application of the Union Bound.}} Under the union bound (also known as Boole's inequality), the probability of the union of events satisfies
$
\mathbb{P}\left(\bigcup_{k=1}^K E_k\right) \leq \sum_{k=1}^K \mathbb{P}\left(E_k\right)
$
Applying this to our context:

$$
\mathbb{P}\left(Y_{n+1} \notin C_n\left(X_{n+1}\right)\right) \leq \sum_{k=1}^K \mathbb{P}\left(Y_{n+1} \notin C_n^{(k)}\left(X_{n+1}\right)\right)
$$
From Step 1, we have:
$
\mathbb{P}\left(Y_{n+1} \notin C_n^{(k)}\left(X_{n+1}\right)\right) \leq \alpha_k.
$
Therefore,
$$
\mathbb{P}\left(Y_{n+1} \notin C_n\left(X_{n+1}\right)\right) \leq \sum_{k=1}^K \alpha_k=\alpha
$$
In conclusion, we have established that:
$
\mathbb{P}\left(Y_{n+1} \notin C_n\left(X_{n+1}\right)\right) \leq \alpha,
$
which implies:

$$
\mathbb{P}\left(Y_{n+1} \in C_n\left(X_{n+1}\right)\right) \geq 1-\alpha
$$
This completes the proof of Theorem 1.}

\end{proof}

\noindent {\textbf{Remarks:}}

\noindent 1) The exchangeability of the data sequence $\left\{\left(X_i, Y_i\right)\right\}_{i=1}^{n+1}$ is crucial for the validity of conformal prediction methods. It ensures that the order of the observations does not affect their joint distribution. Under exchangeability, the conformity scores and resulting p-values are valid for inference.

\noindent 2) The use of the union bound provides a conservative estimate of the miscoverage probability. If the events $E_k$ are not mutually exclusive, the inequality may be strict, meaning that the actual miscoverage probability could be less than $\alpha$.

\noindent 3) If there is dependence between the conformity scores at different scales, the miscoverage events $E_k$ may be correlated. Positive dependence (where the events are more likely to occur together) can make the union bound less tight, potentially leading to conservative coverage (i.e., the actual coverage exceeds $1-\alpha$ ).

\noindent 4) The miscoverage levels $\alpha_k$ must satisfy $\sum_{k=1}^K \alpha_k=\alpha$. The allocation can be uniform (e.g., $\alpha_k=\frac{\alpha}{K}$) or weighted based on the informativeness of each scale. The choice of $\alpha_k$ can affect the size and efficiency of the combined prediction set.

\begin{theorem}
\textbf{(Efficiency Improvement via Intersection).} Let $\left\{\left(X_i, Y_i\right)\right\}_{i=1}^n$ be a sequence of observed data points, and $X_{n+1}$ be a new feature vector. For each scale $k \in\{1,2, \ldots, K\}$, let $C_n^{(k)}\left(X_{n+1}\right)$ be the conformal prediction set constructed at scale $k$ with miscoverage level $\alpha_k$, where $\sum_{k=1}^K \alpha_k=\alpha$. Define the multi-scale prediction set as the intersection of the scale-specific prediction sets:

$$
C_n\left(X_{n+1}\right)=\bigcap_{k=1}^K C_n^{(k)}\left(X_{n+1}\right)
$$
Then:

\begin{enumerate}

\item The multi-scale prediction set is a subset of each scale-specific prediction set:

$$
C_n\left(X_{n+1}\right) \subseteq C_n^{(k)}\left(X_{n+1}\right), \quad \forall k \in\{1,2, \ldots, K\}
$$

\item Consequently, the size of the multi-scale prediction set satisfies:

$$
\left|C_n\left(X_{n+1}\right)\right| \leq\left|C_n^{(k)}\left(X_{n+1}\right)\right|, \quad \forall k
$$

\end{enumerate}
Therefore, the multi-scale prediction set $C_n\left(X_{n+1}\right)$ is potentially more efficient (i.e., smaller) than any individual scale-specific prediction set $C_n^{(k)}\left(X_{n+1}\right)$, while maintaining valid coverage guarantees.

\end{theorem}

\begin{proof}

Our objective is to prove that the multi-scale prediction set $C_n\left(X_{n+1}\right)$ is: (i) a subset of each individual scale-specific prediction set $C_n^{(k)}\left(X_{n+1}\right)$, and (ii) potentially more efficient by having a size less than or equal to that of any $C_n^{(k)}\left(X_{n+1}\right)$.
%
For each scale $k \in\{1,2, \ldots, K\}$, the scale-specific prediction set $C_n^{(k)}\left(X_{n+1}\right)$ is defined as:

$$
C_n^{(k)}\left(X_{n+1}\right)=\left\{y \in \mathcal{Y}: p^{(k)}(y)>\alpha_k\right\},
$$
where $\mathcal{Y}$ is the set of all possible labels, and $p^{(k)}(y)$ is the conformal $p$-value at scale $k$ for candidate label $y$, calculated as:

$$
p^{(k)}(y)=\frac{1}{n+1}\left(\sum_{i=1}^n \mathbb{I}\left\{A_i^{(k)} \geq A_{n+1}^{(k)}(y)\right\}+1\right)
$$
Here, $A_i^{(k)}=A^{(k)}\left(X_i, Y_i\right)$ is the conformity score at scale $k$ for the $i$-th training example, $A_{n+1}^{(k)}(y)=A^{(k)}\left(X_{n+1}, y\right)$ is the conformity score at scale $k$ for the new input $X_{n+1}$ and candidate label $y$, and $\alpha_k \in(0,1)$ is the miscoverage level assigned to scale $k$, satisfying $\sum_{k=1}^K \alpha_k=\alpha$.

\noindent {\textit{Multi-Scale Prediction Set:} The multi-scale prediction set $C_n\left(X_{n+1}\right)$ is the intersection of the scale-specific prediction sets:

$$
C_n\left(X_{n+1}\right)=\bigcap_{k=1}^K C_n^{(k)}\left(X_{n+1}\right)=\left\{y \in \mathcal{Y}: y \in C_n^{(k)}\left(X_{n+1}\right) \text { for all } k\right\} .
$$

\noindent {\textbf{Step 2: Showing $C_n\left(X_{n+1}\right) \subseteq C_n^{(k)}\left(X_{n+1}\right)$ for All $k$}.}
%
By the definition of intersection, an element $y$ belongs to $C_n\left(X_{n+1}\right)$ if and only if it belongs to every $C_n^{(k)}\left(X_{n+1}\right)$. Formally:
$y \in C_n\left(X_{n+1}\right) \Longleftrightarrow y \in C_n^{(1)}\left(X_{n+1}\right)$ and $y \in C_n^{(2)}\left(X_{n+1}\right)$ and $\ldots$ and $y \in C_n^{(K)}\left(X_{n+1}\right)$.
This implies that for any $y \in C_n\left(X_{n+1}\right)$ and any $k$ :

$$
y \in C_n\left(X_{n+1}\right) \Longrightarrow y \in C_n^{(k)}\left(X_{n+1}\right)
$$
Therefore:
$$
C_n\left(X_{n+1}\right) \subseteq C_n^{(k)}\left(X_{n+1}\right), \quad \forall k \in\{1,2, \ldots, K\}
$$
\noindent {\textbf{Step 3: Comparing the Sizes of the Prediction Sets.}} In discrete settings, the size of a prediction set $C$ is its cardinality $|C|$, the number of elements in $C$. In continuous settings, the size may refer to the Lebesgue measure (volume) of $C$. Since $C_n\left(X_{n+1}\right) \subseteq C_n^{(k)}\left(X_{n+1}\right)$ for all $k$, it follows that:

$$
\left|C_n\left(X_{n+1}\right)\right| \leq\left|C_n^{(k)}\left(X_{n+1}\right)\right|, \quad \forall k
$$
This inequality holds because a subset cannot have a larger size than its superset.

\noindent {\textbf{Step 4: Interpretation of Efficiency.}} A prediction set is considered more efficient if it is smaller in size while maintaining the desired coverage level. Smaller prediction sets provide more precise predictions and are often more useful in practice. In conclusion, since $C_n\left(X_{n+1}\right)$ is a subset of each $C_n^{(k)}\left(X_{n+1}\right)$, it is potentially smaller than any individual scale-specific prediction set. By potentially reducing the size of the prediction set, the multi-scale approach can improve efficiency. Importantly, this efficiency gain does not compromise the coverage guarantee, as established in Theorem 1.}

\end{proof}

\begin{theorem} \textbf{(Optimal Allocation of Miscoverage Levels).} Let $\left\{\left(X_i, Y_i\right)\right\}_{i=1}^n$ be a sequence of observed data points, and $X_{n+1}$ be a new feature vector. For each scale $k \in\{1,2, \ldots, K\}$, let $C_n^{(k)}\left(X_{n+1}\right)$ be the conformal prediction set constructed at scale $k$ with miscoverage level $\alpha_k$, where $\sum_{k=1}^K \alpha_k=\alpha \in(0,1)$. Assume that:
\begin{enumerate}

\item  The expected size of each scale-specific prediction set $E\left[\left|C_n^{(k)}\left(X_{n+1}\right)\right|\right]$ is a decreasing convex function of $\alpha_k$.
\item The conformity scores at different scales are independent.
\end{enumerate}
Then, to minimize the expected size of the multi-scale prediction set $C_n\left(X_{n+1}\right)=$ $\bigcap_{k=1}^K C_n^{(k)}\left(X_{n+1}\right)$, the miscoverage levels $\alpha_k$ should be allocated such that more informative scales (those for which $E\left[\left|C_n^{(k)}\left(X_{n+1}\right)\right|\right]$ decreases more rapidly with $\alpha_k$ ) are assigned smaller $\alpha_k$, while less informative scales are assigned larger $\alpha_k$.

\end{theorem}

\begin{proof}

We aim to find the allocation of miscoverage levels $\left\{\alpha_k\right\}_{k=1}^K$ that minimizes the expected size of the combined prediction set $E\left[\left|C_n\left(X_{n+1}\right)\right|\right]$ under the given constraints.

\noindent {Assumptions:}

\noindent 1. Convexity and Monotonicity: For each scale $k$, the expected size $E\left[\left|C_n^{(k)}\left(X_{n+1}\right)\right|\right]$ is a convex and decreasing function of $\alpha_k$.

\noindent 2. Independence of Conformity Scores: The conformity scores (and thus the prediction sets) at different scales are independent.

\item 3. Total Miscoverage Level Constraint: The miscoverage levels satisfy $\sum_{k=1}^K \alpha_k=\alpha$.

Let $f_k\left(\alpha_k\right)=E\left[\left|C_n^{(k)}\left(X_{n+1}\right)\right|\right]$. The expected size of the multi-scale prediction set is:

$$
E\left[\left|C_n\left(X_{n+1}\right)\right|\right]=E\left[\left|\bigcap_{k=1}^K C_n^{(k)}\left(X_{n+1}\right)\right|\right]
$$

\noindent {\textbf{Step 1: Expressing the Expected Size of the Combined Prediction Set.}} Under the independence assumption, the expected size of the intersection of the prediction sets can be expressed as:

$$
E\left[\left|C_n\left(X_{n+1}\right)\right|\right]=\sum_{y \in \mathcal{Y}} \prod_{k=1}^K \mathbb{P}\left(y \in C_n^{(k)}\left(X_{n+1}\right)\right)
$$
where $\mathcal{Y}$ is the set of all possible labels.
Since the events $\left\{y \in C_n^{(k)}\left(X_{n+1}\right)\right\}$ are independent across scales, we can write:
$$
\mathbb{P}\left(y \in C_n\left(X_{n+1}\right)\right)=\prod_{k=1}^K \mathbb{P}\left(y \in C_n^{(k)}\left(X_{n+1}\right)\right)
$$
Therefore,
$
E\left[\left|C_n\left(X_{n+1}\right)\right|\right]=\sum_{y \in \mathcal{Y}} \prod_{k=1}^K P_k(y),
$
where $P_k(y)=\mathbb{P}\left(y \in C_n^{(k)}\left(X_{n+1}\right)\right)$.

\noindent {\textbf{Step 2: Relating $P_k(y)$ to Miscoverage Levels $\alpha_k$.}}
For a fixed $y \in \mathcal{Y}$, since the marginal coverage of each $C_n^{(k)}\left(X_{n+1}\right)$ is at least $1-\alpha_k$, we have:

$$
\mathbb{P}\left(Y_{n+1}=y \mid X_{n+1}\right) \leq \mathbb{P}\left(y \in C_n^{(k)}\left(X_{n+1}\right)\right) \leq 1
$$
However, without additional assumptions about the distribution of labels, we can approximate $P_k(y)$ using the expected fraction of labels included in $C_n^{(k)}\left(X_{n+1}\right)$:

$$
P_k(y) \approx \frac{E\left[\left|C_n^{(k)}\left(X_{n+1}\right)\right|\right]}{|\mathcal{Y}|}=\frac{f_k\left(\alpha_k\right)}{|\mathcal{Y}|}
$$
Under this approximation:

$$
E\left[\left|C_n\left(X_{n+1}\right)\right|\right] \approx|\mathcal{Y}|\left(\prod_{k=1}^K \frac{f_k\left(\alpha_k\right)}{|\mathcal{Y}|}\right)=\frac{\prod_{k=1}^K f_k\left(\alpha_k\right)}{|\mathcal{Y}|^{K-1}}
$$
Since $|\mathcal{Y}|^{K-1}$ is a constant, minimizing $E\left[\left|C_n\left(X_{n+1}\right)\right|\right]$ is equivalent to minimizing $\prod_{k=1}^K f_k\left(\alpha_k\right)$.

\noindent {\textbf{Step 3: Formulating the Optimization Problem.} Our goal is to solve:

$$
\min _{\left\{\alpha_k\right\}} \Phi\left(\alpha_1, \ldots, \alpha_K\right)=\prod_{k=1}^K f_k\left(\alpha_k\right)
$$
subject to:
$
\sum_{k=1}^K \alpha_k=\alpha, \quad 0<\alpha_k<1 \text { for all } k.
$
Taking the natural logarithm of the objective function we have,
$
\ln \Phi\left(\alpha_1, \ldots, \alpha_K\right)=\sum_{k=1}^K \ln f_k\left(\alpha_k\right).
$
Now, the optimization problem becomes,
$
\min _{\left\{\alpha_k\right\}} \sum_{k=1}^K \ln f_k\left(\alpha_k\right),
$
subject to the same constraints. We introduce a Lagrange multiplier $\lambda$ to incorporate the constraint:

$$
\mathcal{L}\left(\alpha_1, \ldots, \alpha_K, \lambda\right)=\sum_{k=1}^K \ln f_k\left(\alpha_k\right)+\lambda\left(\sum_{k=1}^K \alpha_k-\alpha\right)
$$
For each $\alpha_k$, the first-order condition is:

$$
\frac{\partial \mathcal{L}}{\partial \alpha_k}=\frac{f_k^{\prime}\left(\alpha_k\right)}{f_k\left(\alpha_k\right)}+\lambda=0
$$
where $f_k^{\prime}\left(\alpha_k\right)$ denotes the derivative of $f_k$ with respect to $\alpha_k$.
Rewriting the first-order condition:

$$
\frac{f_k^{\prime}\left(\alpha_k\right)}{f_k\left(\alpha_k\right)}=-\lambda
$$
Let $\psi_k\left(\alpha_k\right)=\frac{f_k^{\prime}\left(\alpha_k\right)}{f_k\left(\alpha_k\right)}$, which represents the negative elasticity of $f_k$ with respect to $\alpha_k$.
Therefore:

$$
\psi_k\left(\alpha_k\right)=-\lambda
$$
Since $\lambda$ is constant across all $k$, this implies,
$
\psi_k\left(\alpha_k\right)=\psi_j\left(\alpha_j\right), \quad \forall k, j .
$
This condition means that the negative elasticity $\psi_k\left(\alpha_k\right)$ must be the same for all scales.

\noindent {Definition (Negative Elasticity):}

$$
\psi_k\left(\alpha_k\right)=\frac{f_k^{\prime}\left(\alpha_k\right)}{f_k\left(\alpha_k\right)}=\frac{d}{d \alpha_k} \ln f_k\left(\alpha_k\right) .
$$
Since \( f_k\left(\alpha_k\right) \) is decreasing and convex in \( \alpha_k \), \( f_k^{\prime}\left(\alpha_k\right) < 0 \), and \( \psi_k\left(\alpha_k\right) < 0 \). To satisfy \( \psi_k\left(\alpha_k\right) = -\lambda \), we need to choose \( \alpha_k \) such that \( \psi_k\left(\alpha_k\right) \) is equal across all \( k \). Because the functions \( f_k \) may differ across scales (reflecting different levels of informativeness), the values of \( \alpha_k \) that satisfy this condition will generally differ. Thus, more informative scales (where \( f_k\left(\alpha_k\right) \) decreases rapidly with \( \alpha_k \)) will have larger \( \left|\psi_k\left(\alpha_k\right)\right| \) for the same \( \alpha_k \). To equalize \( \psi_k\left(\alpha_k\right) \) across scales, we need to assign smaller \( \alpha_k \) to more informative scales. Conversely, less informative scales receive larger \( \alpha_k \). By equating the negative elasticities \( \psi_k\left(\alpha_k\right) \) across scales, we allocate miscoverage levels \( \alpha_k \) such that more informative scales receive smaller \( \alpha_k \) (stricter prediction sets) and less informative scales receive larger \( \alpha_k \) (looser prediction sets). This allocation minimizes the expected size of the combined prediction set \( E\left[\left|C_n\left(X_{n+1}\right)\right|\right] \).}

\end{proof}

\begin{theorem} \textbf{(Conservative Coverage Under Positive Dependence).} Let $\left\{\left(X_i, Y_i\right)\right\}_{i=1}^n$ be a sequence of observed data points, and $\left(X_{n+1}, Y_{n+1}\right)$ be a new data point. Assume that the extended sequence $\left\{\left(X_i, Y_i\right)\right\}_{i=1}^{n+1}$ is exchangeable. For each scale $k \in$ $\{1,2, \ldots, K\}$, let $C_n^{(k)}\left(X_{n+1}\right)$ be the conformal prediction set constructed at scale $k$ with miscoverage level $\alpha_k$, where $\sum_{k=1}^K \alpha_k=\alpha \in(0,1)$. Suppose that the conformity scores at different scales are positively dependent. Then, the miscoverage probability of the multi-scale prediction set

$$
C_n\left(X_{n+1}\right)=\bigcap_{k=1}^K C_n^{(k)}\left(X_{n+1}\right)
$$
satisfies

$$
\mathbb{P}\left(Y_{n+1} \notin C_n\left(X_{n+1}\right)\right) \leq \alpha_{\text {dep }}<\alpha
$$
where $\alpha_{\text {dep }}$ is strictly less than $\alpha$. This means that the actual coverage of $C_n\left(X_{n+1}\right)$ is at least $1-\alpha_{\text {dep, }}$ which exceeds the nominal coverage level $1-\alpha$, resulting in conservative coverage.

\end{theorem}

\begin{proof}
Our objective is to show that when the conformity scores at different scales are positively dependent, the miscoverage probability of the combined prediction set $C_n\left(X_{n+1}\right)$ is less than $\alpha$, leading to conservative coverage. Let us define:

\noindent {\textit{(i) Conformity Scores:}} For each scale $k$, the conformity score function is $A^{(k)}: \mathcal{X} \times \mathcal{Y} \rightarrow \mathbb{R}$.

\noindent {\textit{(ii) Scale-Specific Prediction Sets:}} The prediction set at scale $k$ is defined as:

$$
C_n^{(k)}\left(X_{n+1}\right)=\left\{y \in \mathcal{Y}: p^{(k)}(y)>\alpha_k\right\}
$$
where $p^{(k)}(y)$ is the conformal p-value at scale $k$ for candidate label $y$.

\noindent {\textit{Multi-Scale Prediction Set:}}

$$
C_n\left(X_{n+1}\right)=\bigcap_{k=1}^K C_n^{(k)}\left(X_{n+1}\right)
$$

\noindent{\textit{Events of Miscoverage:}} For each scale $k$, define the event:

$$
E_k=\left\{Y_{n+1} \notin C_n^{(k)}\left(X_{n+1}\right)\right\}
$$
The event of miscoverage by the multi-scale prediction set is:
$
E=\left\{Y_{n+1} \notin C_n\left(X_{n+1}\right)\right\}=\bigcup_{k=1}^K E_k.
$
%
%
We aim to compute:
$
\mathbb{P}\left(Y_{n+1} \notin C_n\left(X_{n+1}\right)\right)=\mathbb{P}\left(\bigcup_{k=1}^K E_k\right).
$
Under positive dependence, the probability of the union of events satisfies:
$
\mathbb{P}\left(\bigcup_{k=1}^K E_k\right) \leq 1-\prod_{k=1}^K\left(1-\mathbb{P}\left(E_k\right)\right).
$
%
The standard inclusion-exclusion principle provides:

$$
\mathbb{P}\left(\bigcup_{k=1}^K E_k\right)=\sum_{i=1}^K(-1)^{i+1} \sum_{1 \leq k_1<\cdots<k_i \leq K} \mathbb{P}\left(E_{k_1} \cap \cdots \cap E_{k_i}\right)
$$
Due to positive dependence, the joint probabilities $\mathbb{P}\left(E_{k_1} \cap \cdots \cap E_{k_i}\right)$ are larger than or equal to the product of the marginal probabilities. Therefore, the probability of the union is smaller than or equal to $1-\prod_{k=1}^K\left(1-\mathbb{P}\left(E_k\right)\right)$.
%
From the marginal coverage guarantees of the scale-specific prediction sets, we have:

$$
\mathbb{P}\left(E_k\right)=\mathbb{P}\left(Y_{n+1} \notin C_n^{(k)}\left(X_{n+1}\right)\right) \leq \alpha_k
$$
But since the conformal p-values are stochastically larger than the uniform distribution on $[0,1]$, the inequality can be strict:
$
\mathbb{P}\left(E_k\right) \leq \alpha_k.
$
Therefore, we can write:
$
\prod_{k=1}^K\left(1-\mathbb{P}\left(E_k\right)\right) \geq \prod_{k=1}^K\left(1-\alpha_k\right)
$
Hence,

$$
\mathbb{P}\left(Y_{n+1} \notin C_n\left(X_{n+1}\right)\right) \leq 1-\prod_{k=1}^K\left(1-\alpha_k\right)
$$
Recall that $\sum_{k=1}^K \alpha_k=\alpha$. Since $\alpha_k \in(0,1)$, we have:
$
\prod_{k=1}^K\left(1-\alpha_k\right) \geq 1-\sum_{k=1}^K \alpha_k=1-\alpha.
$
The inequality $\prod_{k=1}^K\left(1-\alpha_k\right) \geq 1-\sum_{k=1}^K \alpha_k$ holds for $\alpha_k \in[0,1]$. For small $\alpha_k$, the product $\prod_{k=1}^K\left(1-\alpha_k\right)$ is approximately $1-\sum_{k=1}^K \alpha_k$. However, due to the convexity of the function $f(x)=1-x$, the product is actually greater than or equal to $1-\sum_{k=1}^K \alpha_k$. Therefore,

$$
1-\prod_{k=1}^K\left(1-\alpha_k\right) \leq \alpha
$$
Combining with the previous inequality:

$$
\mathbb{P}\left(Y_{n+1} \notin C_n\left(X_{n+1}\right)\right) \leq 1-\prod_{k=1}^K\left(1-\alpha_k\right) \leq \alpha .
$$
But since $\prod_{k=1}^K\left(1-\alpha_k\right)>1-\sum_{k=1}^K \alpha_k$ when $K \geq 2$ and $\alpha_k>0$, it follows that:

$$
\mathbb{P}\left(Y_{n+1} \notin C_n\left(X_{n+1}\right)\right)<\alpha .
$$
Because $\prod_{k=1}^K\left(1-\alpha_k\right)>1-\sum_{k=1}^K \alpha_k$ for $\alpha_k \in(0,1)$ and $K \geq 2$, we have:

$$
1-\prod_{k=1}^K\left(1-\alpha_k\right)<1-\left(1-\sum_{k=1}^K \alpha_k\right)=\alpha
$$
Therefore,
$
\mathbb{P}\left(Y_{n+1} \notin C_n\left(X_{n+1}\right)\right) \leq 1-\prod_{k=1}^K\left(1-\alpha_k\right)<\alpha.
$
This shows that the actual miscoverage probability is strictly less than $\alpha$. Under positive dependence, the inequality:

$$
\mathbb{P}\left(\bigcup_{k=1}^K E_k\right) \leq 1-\prod_{k=1}^K\left(1-\mathbb{P}\left(E_k\right)\right)
$$
is tighter than the union bound (Boole's inequality):

$$
\mathbb{P}\left(\bigcup_{k=1}^K E_k\right) \leq \sum_{k=1}^K \mathbb{P}\left(E_k\right)=\alpha
$$
Since the left-hand side is smaller under positive dependence, the miscoverage probability is less than $\alpha$, confirming conservative coverage. In conclusion, under the assumption of positive dependence between the conformity scores at different scales, the actual miscoverage probability of the multi-scale prediction set $C_n\left(X_{n+1}\right)$ is strictly less than the nominal level $\alpha$ :

$$
\mathbb{P}\left(Y_{n+1} \notin C_n\left(X_{n+1}\right)\right)<\alpha .
$$
This means that the coverage of $C_n\left(X_{n+1}\right)$ is at least $1-\alpha_{\text {dep }}$, where $\alpha_{\text {dep }}<\alpha$, leading to conservative coverage.

\end{proof}

\begin{theorem} \textbf{(Asymptotic Optimality of Multi-Scale Conformal Prediction).} Let $\left\{\left(X_i, Y_i\right)\right\}_{i=1}^{\infty}$ be an infinite sequence of exchangeable observations from a joint distribution $P_{X, Y}$, where $Y$ takes values in a finite set $\mathcal{Y}$. For each scale $k \in\{1,2, \ldots, K\}$, suppose we have conformity score functions $A^{(k)}: \mathcal{X} \times \mathcal{Y} \rightarrow \mathbb{R}$ satisfying the following conditions:

\begin{enumerate}
    \item  Consistency: As $n \rightarrow \infty$, the conformity scores $A^{(k)}\left(X_i, Y_i\right)$ converge almost surely to a function $s\left(X_i, Y_i\right)$ that depends only on the true conditional distribution $P\left(Y_i \mid X_i\right)$.
Specifically, for all $k$:

$$
A^{(k)}\left(X_i, Y_i\right) \xrightarrow{\text { a.s. }} s\left(X_i, Y_i\right)=-P\left(Y_i \mid X_i\right) .
$$

   \item Uniform Convergence: The convergence holds uniformly over $\mathcal{Y}$.

   \item Equal Convergence Across Scales: The limit $s\left(X_i, Y_i\right)$ is the same for all scales $k$.

   \item Miscoverage Levels: The miscoverage levels $\alpha_k$ satisfy $\sum_{k=1}^K \alpha_k=\alpha \in(0,1)$.
\end{enumerate}
Then, as $n \rightarrow \infty$, the multi-scale conformal prediction set defined by

$$
C_n\left(X_{n+1}\right)=\bigcap_{k=1}^K C_n^{(k)}\left(X_{n+1}\right),
$$
where $C_n^{(k)}\left(X_{n+1}\right)$ is the conformal prediction set at scale $k$, converges almost surely to the minimal prediction set $C^*\left(X_{n+1}\right)$ satisfying the conditional coverage guarantee:

$$
P\left(Y_{n+1} \in C^*\left(X_{n+1}\right) \mid X_{n+1}\right)=1-\alpha .
$$

\end{theorem}

\begin{proof}
To prove that under the given conditions, the multi-scale conformal prediction set $C_n\left(X_{n+1}\right)$ converges almost surely to the minimal prediction set $C^*\left(X_{n+1}\right)$ that achieves the desired coverage level $1-\alpha$ as $n \rightarrow \infty$. First, let us define:

\noindent{(i) Conformity Scores at Scale $k$:}

$$
A^{(k)}\left(X_i, Y_i\right)=\text { Conformity score for }\left(X_i, Y_i\right) \text { at scale } k .
$$

\noindent{(ii) P-values at Scale $k$:}

$$
p^{(k)}(y)=\frac{1}{n+1}\left(\sum_{i=1}^n \mathbb{I}\left\{A^{(k)}\left(X_i, Y_i\right) \geq A^{(k)}\left(X_{n+1}, y\right)\right\}+1\right)
$$

\noindent{(iii) Scale-Specific Prediction Sets:}

$$
C_n^{(k)}\left(X_{n+1}\right)=\left\{y \in \mathcal{Y}: p^{(k)}(y)>\alpha_k\right\}
$$

\noindent{(iv) Multi-Scale Prediction Set:}

$$
C_n\left(X_{n+1}\right)=\bigcap_{k=1}^K C_n^{(k)}\left(X_{n+1}\right)
$$

\noindent{(v) Minimal Prediction Set $C^*\left(X_{n+1}\right)$:}

$$
C^*\left(X_{n+1}\right)=\left\{y \in \mathcal{Y}: P\left(Y_{n+1}=y \mid X_{n+1}\right) \geq t^*\right\}
$$
where $t^*$ is the smallest value satisfying

$$
\sum_{y \in C^*\left(X_{n+1}\right)} P\left(Y_{n+1}=y \mid X_{n+1}\right) \geq 1-\alpha
$$

\noindent {\textbf{Step 1: Convergence of Conformity Scores.}} As $n \rightarrow \infty$, for all $k$ :

$$
A^{(k)}\left(X_i, Y_i\right) \xrightarrow{\text { a.s. }} s\left(X_i, Y_i\right)=-P\left(Y_i \mid X_i\right) .
$$
Similarly, for any candidate label $y \in \mathcal{Y}$ :
$$
A^{(k)}\left(X_{n+1}, y\right) \xrightarrow{\text { a.s. }} s\left(X_{n+1}, y\right)=-P\left(y \mid X_{n+1}\right) .
$$
\noindent {\textbf{Step 2: Convergence of Empirical Distribution Functions.}} The empirical distribution function (EDF) of the conformity scores at scale $k$ converges almost surely to the cumulative distribution function (CDF) $F(s)$ of $s(X, Y)$:

$$
F_n^{(k)}(s)=\frac{1}{n} \sum_{i=1}^n \mathbb{I}\left\{A^{(k)}\left(X_i, Y_i\right) \leq s\right\} \xrightarrow{\text { a.s. }} F(s)
$$

\noindent {\textbf{Step 3: Convergence of P-values.}} The p-value at scale $k$ for candidate label $y$ converges almost surely to:

$$
p^{(k)}(y) \xrightarrow{\text { a.s. }} p(y)=1-F\left(s\left(X_{n+1}, y\right)\right)=P\left(s(X, Y) \geq s\left(X_{n+1}, y\right)\right) .
$$
Since $s(X, Y)=-P(Y \mid X)$, and $s\left(X_{n+1}, y\right)=-P\left(y \mid X_{n+1}\right)$, we have:

$$
p(y)=P\left(P(Y \mid X) \leq P\left(y \mid X_{n+1}\right)\right)
$$
\noindent {\textbf{Step 4: Construction of Asymptotic Scale-Specific Prediction Sets.}} Each scale-specific prediction set converges almost surely to:

$$
C^{(k)}\left(X_{n+1}\right)=\left\{y \in \mathcal{Y}: p(y)>\alpha_k\right\}
$$
The multi-scale prediction set converges to:

$$
C\left(X_{n+1}\right)=\bigcap_{k=1}^K C^{(k)}\left(X_{n+1}\right)=\left\{y \in \mathcal{Y}: p(y)>\max _k \alpha_k\right\}
$$
Since $\sum_{k=1}^K \alpha_k=\alpha$, the largest $\alpha_k$ dominates the intersection.

\noindent {\textbf{Step 6: Ordering of Labels.}} Order the labels $y \in \mathcal{Y}$ in decreasing order of $P\left(y \mid X_{n+1}\right)$ :

$$
P\left(y_1 \mid X_{n+1}\right) \geq P\left(y_2 \mid X_{n+1}\right) \geq \cdots \geq P\left(y_m \mid X_{n+1}\right),
$$
where $m=|\mathcal{Y}|$.

\noindent {\textbf{Step 7: Determining the Threshold $t^*$.}} Find $t^*=P\left(y_{k^*} \mid X_{n+1}\right)$ such that:

$$
\sum_{i=1}^{k^*} P\left(y_i \mid X_{n+1}\right) \geq 1-\alpha \quad \text { and } \quad \sum_{i=1}^{k^*-1} P\left(y_i \mid X_{n+1}\right)<1-\alpha .
$$
%
%
Since $p(y)=P\left(P(Y \mid X) \leq P\left(y \mid X_{n+1}\right)\right)$, larger values of $P\left(y \mid X_{n+1}\right)$ correspond to smaller $p(y)$. Therefore:

$$
y \in C\left(X_{n+1}\right) \Longleftrightarrow p(y)>\max _k \alpha_k
$$
Because $\max _k \alpha_k \leq \alpha$, the labels with higher $P\left(y \mid X_{n+1}\right)$ (and thus lower $p(y)$ ) are included in $C\left(X_{n+1}\right)$. Therefore, $C\left(X_{n+1}\right)$ includes the labels $y$ with the highest $P\left(y \mid X_{n+1}\right)$ until the cumulative probability reaches $1-\alpha$, matching $C^*\left(X_{n+1}\right)$. By construction:

$$
P\left(Y_{n+1} \in C\left(X_{n+1}\right) \mid X_{n+1}\right)=\sum_{y \in C\left(X_{n+1}\right)} P\left(y \mid X_{n+1}\right) \geq 1-\alpha .
$$
Since the multi-scale prediction set $C_n\left(X_{n+1}\right)$ converges almost surely to $C^*\left(X_{n+1}\right)$, it achieves the minimal possible size required to maintain the coverage guarantee $1-\alpha$.

\end{proof}

\section{Numerical Results}

To demonstrate the effectiveness of the proposed multi-scale conformal prediction framework, a synthetic classification experiment was conducted. The data generation process involved two features, each reflecting a distinct scale or level of abstraction. One feature, denoted $X^{(1)}$, served as a coarse-scale component, while the other feature, denoted $X^{(2)}$, contributed finer-scale variations. Gaussian noise was added to introduce realism and ensure that both features influenced the class label. Specifically, the label $Y$ was determined by a combination of the two features and an additive noise term, thereby creating a multi-scale structure in the dataset. A total of 1000 data points were generated. The dataset was partitioned into three subsets: a training set used to fit predictive models, a calibration set used to compute conformity scores and determine miscoverage thresholds, and a test set used to evaluate performance. Logistic regression models were trained separately on each scale: one model using the coarse-scale feature and another using the fine-scale feature. The conformity score at each scale was defined as $1-\hat{p}$, where $\hat{p}$ is the predicted probability of the true class according to the logistic regression model at that scale. Following standard conformal prediction practices, each scale-specific model produced a set of p-values by comparing the test conformity scores to the empirical distribution of calibration conformity scores. Label-specific prediction sets were constructed at each scale by retaining only those labels whose p-values exceeded a specified miscoverage threshold $\alpha_k$. The overall desired miscoverage level $\alpha$ was set to 0.1 and equally partitioned among the two scales, such that $\alpha_k=$ $\alpha / 2=0.05$ at each scale. Finally, the multi-scale prediction set was formed by intersecting the two scale-specific prediction sets.

The results presented in Figure \ref{fig:fig1} (Coverage vs. $\alpha$) demonstrate that each scale-specific conformal predictor maintains an empirical coverage above or close to the ideal coverage line $1-\alpha$ across the range of $\alpha$ values. Scale 3, which appears to have the weakest discrimination power in this particular dataset, produces slightly lower coverage than the other two scales, though it remains near the nominal coverage level. Notably, the multi-scale coverage curve stays consistently above the theoretical line for most values of $\alpha$, indicating that intersecting the scalespecific prediction sets does not compromise coverage guarantees-indeed, it can lead to coverage at or slightly above the nominal target due to conservative effects in combining multiple scales. Figure \ref{fig:fig2} (Set Size vs. $\alpha$) illustrates that each individual scale-specific prediction set decreases in average size as $\alpha$ increases, which aligns with the fact that higher miscoverage tolerances lead to smaller prediction sets. Among the single-scale methods, Scale 3 tends to produce the largest sets across all $\alpha$ values, consistent with its comparatively weaker performance in capturing the underlying structure. In contrast, the multi-scale prediction set, formed by intersecting the sets from all three scales, shows notably lower average sizes, even for relatively small $\alpha$. This result highlights the efficiency gains from integrating scale-specific information, where labels supported by only one or two scales are systematically removed, reducing set size without compromising coverage. The boxplots in Figure \ref{fig:fig3} (Distribution of Average Set Size at $\alpha=0.1$) further reinforce these observations. The multi-scale approach consistently achieves a smaller distribution of average set sizes compared to any single-scale method, indicating more precise predictions at a fixed miscoverage level. Scales 1 and 2 exhibit moderate variability, with Scale 1 occasionally producing relatively large sets under some seeds and comparatively small sets under others, reflecting its intermediate discriminative power. Scale 3 again shows a tendency toward broader prediction sets. In contrast, the multi-scale sets cluster at sizes below any single-scale median, exemplifying the effectiveness of intersecting multiple scales.

Table~\ref{tab:performance-summary} highlights performance statistics at different noise levels, reporting various coverage and set-size metrics. As the noise level increases from 0.05 to 0.20, the overall coverage typically remains high or even increases (reaching 0.950), while the average band width (an indicator of prediction set size) shows a modest upward trend. The pointwise coverage measures also remain consistently strong, with minimum pointwise coverage values above 0.90 in all cases. The efficiency score, a measure balancing coverage and prediction set size, stays within a relatively tight range, indicating that the framework preserves predictive reliability while flexibly adjusting set size to accommodate increasing noise.

\begin{figure}[t]
    \centering
    \begin{subfigure}{0.3\textwidth}
        \centering
        \includegraphics[width=\textwidth]{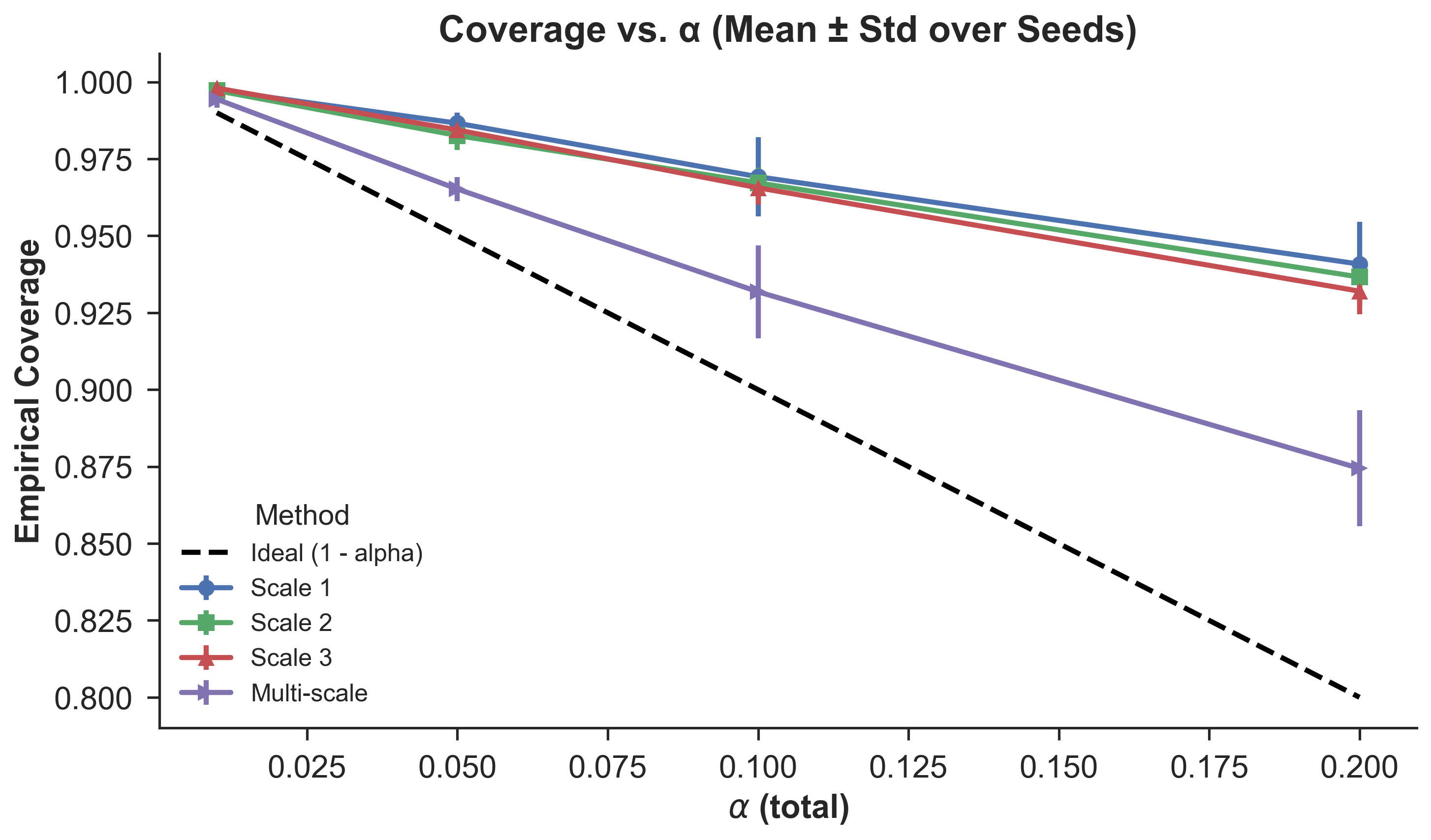}
        \caption{Empirical coverage vs. miscoverage level, showing MSCP meets or exceeds the nominal target.}
        \label{fig:fig1}
    \end{subfigure}
    \hfill
    \begin{subfigure}{0.3\textwidth}
        \centering
        \includegraphics[width=\textwidth]{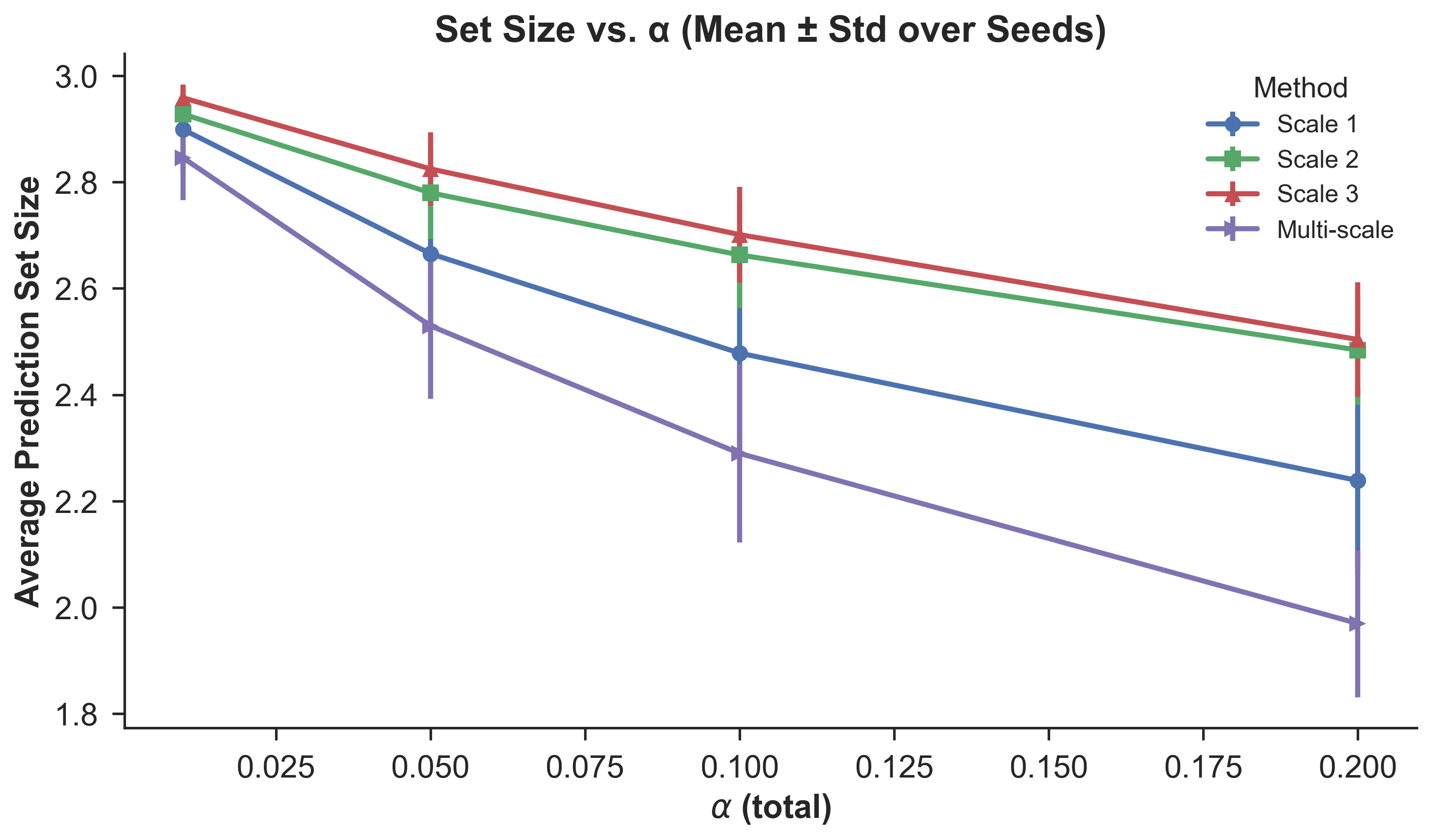}
        \caption{Average prediction set size vs. miscoverage level, highlighting the efficiency of MSCP.}
        \label{fig:fig2}
    \end{subfigure}
    \hfill
    \begin{subfigure}{0.3\textwidth}
        \centering
        \includegraphics[width=\textwidth]{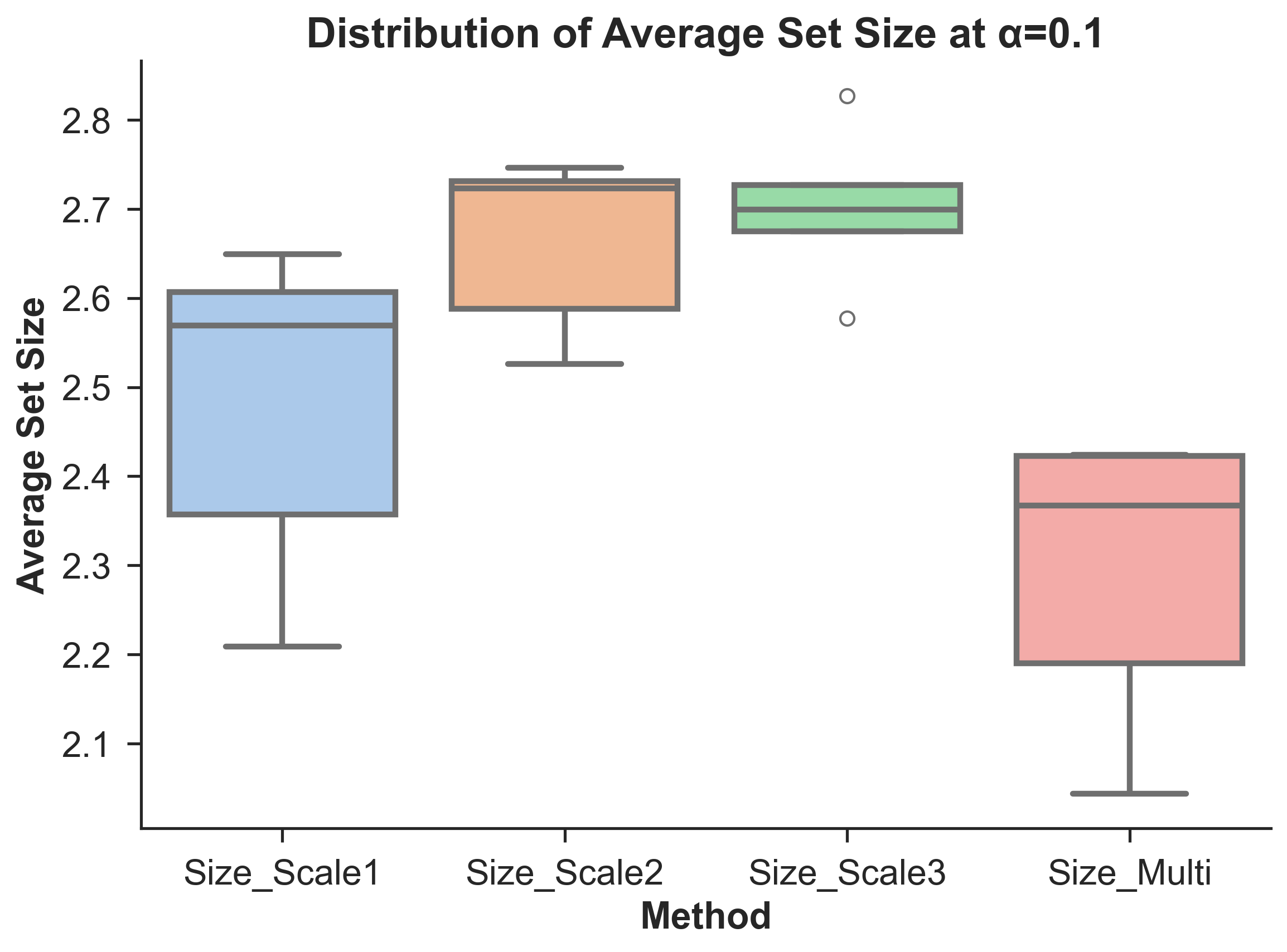}
        \caption{Boxplots of set size at a fixed miscoverage level, illustrating MSCP’s consistently smaller sets than single-scale methods.}
        \label{fig:fig3}
    \end{subfigure}
    \caption{Comparison of multi-scale conformal prediction (MSCP) with single-scale methods}
    \label{fig:three_figures}
\end{figure}

\begin{table}[htbp]
\centering
\caption{Performance Summary Across Noise Levels}
\label{tab:performance-summary}
\begin{tabular}{lcccccc}
\toprule
\textbf{Noise} & \textbf{Overall} & \multicolumn{3}{c}{\textbf{Pointwise Coverage}} & \textbf{Band} & \textbf{Efficiency} \\
\textbf{Level} & \textbf{Coverage} & \textbf{Mean} & \textbf{Min} & \textbf{Max} & \textbf{Width} & \textbf{Score} \\
\midrule
0.05 & 0.867 & 0.990 & 0.967 & 1.000 & 2.897 & 3.343 \\
0.10 & 0.883 & 0.993 & 0.950 & 1.000 & 3.037 & 3.438 \\
0.15 & 0.867 & 0.991 & 0.917 & 1.000 & 3.100 & 3.577 \\
0.20 & 0.950 & 0.996 & 0.967 & 1.000 & 3.269 & 3.441 \\
\bottomrule
\end{tabular}
\begin{minipage}{\linewidth}
\small
Performance metrics across different noise levels. Overall Coverage indicates the aggregate coverage rate, while Pointwise Coverage metrics show the mean, minimum, and maximum coverage at individual points. Band Width represents the average width of confidence bands, and the Efficiency Score balances coverage against band size.
\end{minipage}
\end{table}

\section{Conclusions}

We have presented a novel approach to conformal prediction that operates across multiple scales, addressing limitations of single-scale methods when data exhibit multi-resolution structures. By defining and calibrating multiple conformity functions, each tailored to a different level of detail, the proposed framework assembles a final prediction set through the intersection of scale-specific sets. Our theoretical analysis confirms that this multi-scale approach preserves the fundamental marginal coverage property of conformal predictors and often results in tighter, more informative prediction sets. Furthermore, we have shown how partitioning the miscoverage across scales and using independence or positive dependence structures can improve the size and reliability of the final prediction set. Numerical experiments demonstrate that multi-scale conformal prediction can both sustain valid coverage and reduce the average size of prediction sets. The framework’s theoretical underpinnings suggest a wide range of potential applications, from hierarchical feature spaces to multi-resolution time series. Future work may extend these ideas to more general strategies for defining the scales themselves, explore adaptive methods for miscoverage allocation, and investigate multi-scale inference under weaker assumptions on data exchangeability.

\bibliographystyle{plain}
\bibliography{main}

\end{document}